\documentclass[11pt]{article}
\usepackage{amsthm,amssymb,amsfonts,amsmath,amscd}
\usepackage{setup}

\usepackage{comment}

\usepackage[mono=false]{libertine}
\usepackage[T1]{fontenc}
\usepackage[cmintegrals,libertine]{newtxmath}
\usepackage[cal=euler, calscaled=.95, frak=euler]{mathalfa}
\useosf
\usepackage[font=sf, labelfont={sf,bf}, margin=1cm]{caption}

\newlist{compitem}{itemize}{4}
\setlist[compitem,1]{nolistsep,label=$\bullet$}

\usepackage{hyperref}
\hypersetup{
	colorlinks=true,
		citecolor=blue!60!black,
		linkcolor=red!60!black,
		urlcolor=green!40!black,
		filecolor=yellow!50!black,
	breaklinks=true,
	pdfpagemode=UseNone,
	bookmarksopen=false,
}
\colorlet{link}{red!60!black}

\makeatletter
\newcommand{\labeltext}[3][]{%
    \@bsphack%
    \csname phantomsection\endcsname% in case hyperref is used
    \def\tst{#1}%
    \def\labelmarkup{\textcolor{link}}% How to markup the label itself
    \def\refmarkup{}%
    \ifx\tst\empty\def\@currentlabel{\refmarkup{#2}}{\label{#3}}%
    \else\def\@currentlabel{\refmarkup{#1}}{\label{#3}}\fi%
    \@esphack%
    \labelmarkup{#2}% visible printed text.
}
\makeatother

\usepackage{microtype}
\usepackage[normalem]{ulem}
\usepackage{stmaryrd}

\usepackage{cases}
\newcommand{\N}{\mathbb{N}}

\usepackage{bbm}

\newcommand{\F}{\mathcal{F}}

\newcommand{\X}{\mathcal{X}}

\usepackage{subfig}

\usepackage[numbers,longnamesfirst]{natbib}  %%%% For improved citation style

\usepackage[nameinlink,capitalise]{cleveref}

\crefname{theorem}{Theorem}{Theorems}
\crefname{thm}{Theorem}{Theorems}
\crefname{lem}{Lemma}{Lemmas}
\crefname{remark}{Remark}{Remarks}
\crefname{claim}{Claim}{Claims}
\crefname{conj}{Conjecture}{Conjectures}
\crefname{prop}{Proposition}{Propositions}
\crefname{defn}{Definition}{Definitions}
\crefname{cor}{Corollary}{Corollaries}
\crefname{figure}{Figure}{Figures}

\newcounter{algo}
\crefname{algo}{Algorithm}{Algorithms}
\newcommand{\algohead}[1]{%
    \refstepcounter{algo}% step counter
    \paragraph{\textcolor{red!60!black}{Algorithm \thealgo.}}% print title
    \label{#1}% label title
}

\newcommand*\circled[1]{\tikz[baseline=(char.base)]{
            \node[shape=circle,draw,inner sep=1pt] (char) {#1};}}

\newcommand{\Height}{\mathsf{H}} 
\newcommand{\Ht}{\mathsf{h}}

\newcommand{\Xb}{\mathrm{\mathbf{x}}} 
\newcommand{\Xbh}{\hat{\Xb}} 
\newcommand{\Fh}{\hat{\F}} 
\renewcommand{\X}{\mathrm x} 

\newcommand{\Yb}{\mathrm{\mathbf{y}}}

\renewcommand{\Pr}[1]{\mathbb{P}\left(#1\right)}
\newcommand{\Es}[1]{\mathbb{E}\left[#1\right]}

\newcommand{\ed}{\stackrel{d}{=}}

\newcommand{\intervalle}[4]{\mathopen{#1}#2
                            \mathclose{}\mathpunct{},#3
                            \mathclose{#4}}
\newcommand{\iinterval}[2]{\intervalle\llbracket{#1}{#2}
                               \rrbracket}

\newcommand{\T}{\mathcal{T}} 
\newcommand{\leqst}{\preceq_{\mathrm{st}}}
\newcommand{\geqst}{\succeq_{\mathrm{st}}}

\usepackage{tikz}
\usetikzlibrary{positioning, fit,backgrounds,arrows,shapes,decorations.pathreplacing}

\title{Does freezing impede the growth of random recursive trees?}

\author{
\qquad  
Anna Brandenberger
\thanks{Department of Mathematics, MIT, \textsf{abrande@mit.edu}
}
 \qquad  
Simon Briend
\thanks{Department of Mathematics and Computer Science, Unidistance Suisse, \textsf{simon.briend@unidistance.ch}
} 
\qquad
Hannah Cairns
\thanks{Department of Mathematics and Statistics, McGill University,
\textsf{hannah.cairns@mcgill.ca}
}  \\  
 Robin Khanfir
\thanks{Department of Mathematics and Statistics, McGill University, \textsf{robin.khanfir@mcgill.ca}
} \qquad
Igor Kortchemski 
\thanks{CNRS \& DMA, École normale supérieure, PSL University, 75005 Paris, France, \textsf{igor.kortchemski@math.cnrs.fr}
}
}
\date{}

\begin{document}

\maketitle

\begin{abstract}
Uniform attachment with freezing is an extension of the classical model of random recursive trees, in which trees are recursively built by attaching new vertices
to old ones. In the model of uniform attachment with freezing, vertices are allowed to freeze, in the sense that new vertices cannot be
attached to already frozen ones. We study the impact of removing attachment and/or freezing steps on the height of the trees. We show in particular that removing an attachment step can increase the expected height, and that freezing cannot substantially decrease the height of random recursive trees. Our methods are based on coupling arguments.
\end{abstract}

\section{Introduction}

We are interested in the height of random recursive trees built using  uniform attachment with freezing, which is a model recently introduced by \citet{BBKK23+}. The motivation of the model stems from contact tracing in SIR epidemic models. More generally, the impact of freezing on dynamically-built random graph models is a natural mathematical question.
The work of \citet{BBKK25+}
 studies scaling limits
in the specific regime where the number of non-frozen vertices roughly evolves as the total number of
vertices to a given power. Based on results of \citet{BBKK23+},  \citet{KKS25+} completes the study of the height of the infection tree of a SIR epidemic on the complete graph using uniform attachment with freezing.

This model is defined as follows. Given a \emph{choice sequence}
$\Xb=(\X_i)_{1 \leq i \leq m} \in \{-1,+1\}^m$, we set $S_{0}(\Xb) \coloneqq1$ and for every $j \in \iinterval{1}{m}$, let
\begin{equation}
\label{eq:Sj}S_j(\Xb) \coloneqq 1+ \sum_{i=1}^j \X_i\, ; \quad \text{ and } \quad \tau(\Xb) \coloneqq \min \{j \in \iinterval{1}{m} : S_{j}(\Xb)=0\}
\end{equation}
with the convention $\min \varnothing = \infty$. 
Starting from a unique active vertex, defined to be the root, we recursively build random rooted trees by reading the elements of the sequence one after the other, by applying a ``uniform freezing'' step when reading $-1$ (which amounts to freezing an active vertex chosen uniformly at random) and a ``uniform attachment'' step when reading $+1$ (which amounts to attaching a new vertex to an active vertex chosen uniformly at random). Note that $\tau(\Xb)$, if finite, represents the first time at which all vertices in the tree are frozen. For every $j \in \iinterval{1}{m}$ we denote by $\T_j(\Xb)$ the random rooted tree  built in this fashion after reading the first $j$ elements of $\Xb$ (see \cref{sec:background} for a precise definition) and let $\T(\Xb)=\T_m(\Xb)$ be the last tree. We denote by $\Height(\T(\Xb))$ the height of $\T(\Xb)$ and $\Ht(v)$ the height of a vertex $v \in \T(\Xb)$. 

Finally, we define
\[
\mathbb{X}_n= \bigcup_{m \geq n} \{\Xb= (\X_i)_{1 \leq i \leq m} \in \{-1,+1\}^m : \# \{ i \in \iinterval{1}{m} : \X_i=+1\}=n \textrm{ and } \tau(\Xb) \geq m   \},
\]
so that trees of the form $\mathcal{T}(\Xb)$ with $\Xb \in \mathbb{X}_n$ are all possible $n$-edge trees built by uniform attachment with freezing. To simplify notation, we write $(+1)^m=(+1,+1,\ldots,+1)$ where $+1$ appears $m$ times and similarly define $(-1)^m$. Also, we write $\Xb \mathrm{\mathbf{y}}$ for the concatenation of $\mathrm{\mathbf{x}}$ and $\mathrm{\mathbf{y}}$.

In this note, we address a question of \citet{BBKK23+} on whether $\Height(\T(\Xb))$ always stochastically dominates the height of the uniform attachment tree, also known as the random recursive tree. 
Our main result (\cref{thm:main} below) states that, roughly speaking, the asymptotic order of magnitude of $\Height(\T(\Xb))$ for any $\Xb \in \mathbb{X}_{n}$ is at least $e \log (n)$ with high probability, which is the order of magnitude of the height of a random recursive tree $\mathcal{R}_n$ with $n$ edges. This can informally be rephrased by saying that freezing random trees doesn't asymptotically shrink them. 

At first glance this may seem quite obvious in light of the results of \citet{BBKK25+}: there it is shown (under some technical assumptions) that if the number of active vertices evolves as $n^\alpha$, where $n$ is the total number of vertices and $\alpha \in(0,1)$, then the height of the tree is of order $n^{1-\alpha}$, thus indicating that adding more freezing tends to increase the height. However this is specific to this regime, and not true in general. Indeed, it is \emph{not} true that the height of random recursive tree with $n$ edges is stochastically dominated by $\Height(\T(\Xb))$ for any $\Xb \in \mathbb{X}_{n}$ (see \cref{sec:RRT_vs_alternate}). More generally, there is no obvious coupling that would exhibit a stochastic ordering when modifying the choice sequence. Indeed, the expected height of the tree can increase under both removal of a freezing step or removal of an attachment step, as emphasized in the following result.

\begin{prop}
\label{prop:increaseheight}
The following assertions hold.
\begin{enumerate}
\item[(i)] Removing a freezing step immediately followed by an attachment step can increase the expected height of the tree. More precisely, for every $A>0$, there exist $n,m \in \N$ such that for
\begin{equation}\label{eq:remove-freezing}
    \Xb= (+1)^m (-1)^{m-1} (-1,+1)(+1)^n \qquad \textrm{and} \qquad \hat{\Xb}= (+1)^m (-1)^{m-1} (+1)^n,
\end{equation}
we have $ \Es{\Height(\T(\hat{\Xb}))} \geq \Es{\Height(\T(\Xb))}+A$. 

\quad In particular, if $\bar{\Xb}=(+1)^m (-1)^{m-1} (+1)^{n+1}$ then we have $ \Es{\Height(\T(\bar{\Xb}))} \geq \Es{\Height(\T(\Xb))}+A$, so removing a freezing step can increase the expected height of the tree.
\item[(ii)] Removing an attachment step immediately followed by a freezing step can increase the expected height of the tree. More precisely, for every $A>0$, there exist $n,m \in \N$ such that for
\begin{equation}\label{eq:remove-notfirst}
    \Xb= (+1)^m (-1)^{m-1} (+1,-1)(+1)^n \qquad \textrm{and} \qquad \hat{\Xb}= (+1)^m (-1)^{m-1} (+1)^n,
\end{equation}
we have $ \Es{\Height(\T(\hat{\Xb}))} \geq \Es{\Height(\T(\Xb))}+A$.
\item[(iii)] Removing an attachment step can increase the expected height of the tree. More precisely, there exists $n \in \N$ such that for
\begin{equation}\label{eq:remove-attach}
\Xb=(+1,+1,-1) (+1)^n\qquad \textrm{and} \qquad \hat{\Xb}=(+1,-1) (+1)^n,
\end{equation}
we have $ \Es{\Height(\T(\hat{\Xb}))} >\Es{\Height(\T(\Xb))}$. 
\end{enumerate}
\end{prop}

Despite the lack of stochastic domination under the above two natural operations on the choice sequence, we are able to show that for any $\Xb \in \mathbb{X}^n$, $\Height(\T(\Xb))$ is asymptotically at least of the same order as $\Height(\mathcal R_n)$, so that adding freezing cannot substantially decrease the height of random recursive trees.

For random recursive trees, it is well known (see~\citet[Corollary 1.3]{AF13} and~\citet[Eq.~(1.3)]{PS22}) that  as $n \rightarrow \infty$,
\begin{equation}
\label{eq:H}
\Height(\mathcal{R}_n) = e \log n - \frac{3}{2} \log \log n +O_{\mathbb{P}}(1), \qquad
\Es{\Height(\mathcal{R}_n)} = e \log n - \frac{3}{2} \log \log n +O(1).
\end{equation} 
The following statement formalizes our claim that the asymptotic lower bound for $\Height(\T(\Xb))$ matches $\Height(\mathcal{R}_n)$ to leading order.

\begin{thm}
\label{thm:main}
We have
\[\min_{\Xb \in \mathbb{X}_n}\Pr{\Height(\T(\Xb)) \geq e \log(n) - 5 \log \log n}  \rightarrow 1. 
\]
as $n \to \infty$. Also, for every $n$ sufficiently large,
$$
\min_{\Xb \in \mathbb{X}_n}\Es{\Height(\T(\Xb))}  \geq  e \log n -5 \log \log n.
$$
\end{thm}

The key to proving the result is the following. Recall that the random recursive tree with $n$ edges corresponds to the sequence $\Xb = (+1)^n$ without any freezing steps, so we can write $\mathcal{R}_n = \T((+1)^n)$. We make use of the following coupling which 
% removes a consecutive attachment and freezing step
\emph{does} stochastically decrease the height of the tree.

\begin{thm}
\label{prop:decreaseheight}
 Let $\Xb= (\X_i)_{1 \leq i \leq m} \in \{-1,+1\}^m$ such that $S_{j}(\Xb)>0$ for every $j \in \iinterval{1}{m-1}$. Suppose there exists $k \in \iinterval{1}{m-1}$ such that $\X_1=\X_2=\cdots=\X_{k}=+1$ and $\X_{k+1}=-1$. Let $\hat{\Xb} = (\X_1, \dots, \X_{k-1}, \X_{k+2}, \dots, \X_m)$ be the sequence obtained from $\Xb$ by removing $\X_k$ and $\X_{k+1}$. Then
\[ \Height(\T(\hat{\Xb})) \leqst \Height(\T(\Xb)).\]
\end{thm}

We establish this stochastic domination via a coupling between the construction of the two random trees, which is relatively general and may be of independent interest.
Using it %this result
recursively allows us to compare the heights of a tree with freezing and of a random recursive tree.
The proof of \cref{prop:decreaseheight} is based on an alternative construction of uniform attachment trees with freezing of \citet{BBKK23+}, based on time-reversal, through
a growth-coalescent process of rooted forests.

\paragraph{Outline.} \cref{sec:examples} studies the influence of local modifications of certain choice sequences on the height, and proves \cref{prop:increaseheight}. In \cref{sec:3}, we first show that removing the first attachment step followed by a freezing step stochastically decreases the height (\cref{prop:decreaseheight}) by using an alternative construction of uniform attachment trees with freezing, based on time-reversal, through
a growth-coalescent process of rooted forests  (\cref{sec:3.1,sec:3.2}). The proof of \cref{thm:main} is then given in \cref{sec:3.3}. \cref{sec:open} concludes the paper with some open questions.

\paragraph{Acknowledgements.} This research was initiated during the Eighteenth Annual Workshop on Probability and Combinatorics at McGill University’s Bellairs Institute in Holetown, Barbados. We are grateful to Bellairs Institute for its hospitality. This material is also based upon work supported by the National Science Foundation under Grant No. DMS-1928930, while some of the authors were in
residence at the Simons Laufer Mathematical Sciences Institute in
Berkeley, California, during the Spring 2025 semester.
AB was supported by NSF GRFP 2141064 and Simons Investigator Award 622132 to Elchanan Mossel. 
RK has been supported by the Natural Sciences and Engineering Research Council of Canada (NSERC) via a Banting postdoctoral fellowship [BPF-198443].

\section{Examples and counter-examples}
\label{sec:examples}

\subsection{The recursive construction of uniform attachment with freezing}
\label{sec:background}

Consider a choice sequence $\Xb= (\X_i)_{1 \leq i \leq m} \in \{-1,+1\}^m$ such that $S_{j}(\Xb)>0$ for every $j \in \iinterval{1}{m-1}$. We now formally construct the random trees $(\T_j(\Xb))_{0 \leq j \leq m}$. These trees will be rooted and vertex-labelled by either ``$a$'' (if the vertex is active) or ``$f$'' (if the vertex is frozen).

\algohead{algo1}
\begin{compitem}
\item Start with the tree $\T_0(\Xb)$ made of a single root vertex labelled $a$.
\item For every  $1 \leq j \leq m$, let $V_j$ be a random uniform active vertex of $\T_{j-1}(\Xb)$, chosen independently from all previous ones. Then,
\begin{compitem}
\item[--] if $\X_j=-1$, build $\T_j(\Xb)$ from $\T_{j-1}(\Xb)$ simply by replacing the label $a$ of $V_j$ with label $f$;
\item[--] if $\X_j=1$, build $\T_j(\Xb)$ from $\T_{j-1}(\Xb)$ by adding an edge  between $V_j$ and a new vertex labelled $a$. 
\end{compitem}
\end{compitem}
We let  $\T(\Xb)=\T_m(\Xb)$ be the last tree.  Observe that for $0 \leq j \leq m$, $S_j(\Xb)$ represents the number of active vertices of $\T_{j}(\Xb)$.

\subsection{The height of the frozen tree does not stochastically dominate the height of the random recursive tree}
\label{sec:RRT_vs_alternate}

Here we make the observation that it is not true in general that adding freezing to a random recursive tree stochastically increases the height.  Recall that we write $\mathcal{R}_n=\T((+1)^n)$, where $(+1)^n\in\mathbb{X}_n$ is the sequence of length $n$ without any freezing steps. We also write $\mathcal{A}_n=\T((+1,-1)^n))$, where $(+1,-1)^n \in \mathbb{X}_n$ is the alternating sequence $(+1,-1,+1,-1,\ldots,+1,-1)$ of length $2n$.

\begin{obs}
For every $n \geq 3$ we have $\Height(\mathcal{A}_n)\not\geqst \Height(\mathcal{R}_n)$. 
\end{obs}

This simply comes from the fact that $\Pr{\Height(\mathcal{R}_n) = 1 } < \Pr{\Height(\mathcal{A}_n) = 1} $. %$\Pr{\Height(\mathcal{R}_n) \leq 1} < \Pr{\Height(\mathcal{A}_n) \leq 1}$. 
Indeed, we clearly have $\Pr{\Height(\mathcal{R}_n) =1} = {1}/ n!$ and  $\Pr{\Height(\mathcal{A}_n )= 1}= {1}/{2^{n-1}}$.

It is interesting to note that although the probability $\Pr{\Height(\mathcal{R}_n) =1} $ is asymptotically much smaller than $\Pr{\Height(\mathcal{A}_n) = 1}$, the random variable $\Height(\mathcal{A}_n)$ is typically much larger than $\Height(\mathcal{R}_n) $. Indeed,
\[
\frac{1}{n}  \Height(\mathcal{A}_n) \rightarrow
% \mathop{\longrightarrow}^{(\P)}_{n \rightarrow\infty} 
\frac{1}{2} 
\qquad \textrm{and} \qquad  \frac{1}{\log n}  \Height(\mathcal{R}_n) \rightarrow
% \mathop{\longrightarrow}^{(\P)}_{n \rightarrow \infty} 
e
\]
in probability as $n \to \infty$, respectively by \citet[Theorem 3 (3)]{BBKK23+} and \eqref{eq:H}.

\subsection{Removing a freezing step  immediately followed by an attachment step can increase the expected height}
\label{subseq:remov-1}

Here we establish \cref{prop:increaseheight}(i). Recall that we take
\begin{equation}\label{eq:Xb-and-Xb-hat-i}
    \Xb= (+1)^m (-1)^{m-1} (-1,+1)(+1)^n \qquad \textrm{and} \qquad \hat{\Xb}= (+1)^m (-1)^{m-1} (+1)^n,    
\end{equation}
    so that $\Xbh$ is $\Xb$ with a consecutive freezing and attachment step removed. 

    We first need some technical estimates.  Let $(\mathcal{R}_{i})_{i \geq 0}$, $(\mathcal{R}^{1}_{i})_{i\geq 0}$ and $(\mathcal{R}_{i}^{2})_{i\geq 0}$ be three independent  random recursive trees without freezing (independent from all the other random variables), indexed in such a way that $ \mathcal{R}_{i}$, $ \mathcal{R}^{1}_{i}$ and $ \mathcal{R}^{2}_{i}$ have $i$ edges.  Finally, let $I_n$ be an independent random variable uniformly distributed on $\{0,\ldots,n\}$.

\begin{lem}
\label{lem:estimates}
%The following assertions hold:
There exist two constants $c_1,c_2>0$ such that the following assertions hold:
\begin{enumerate}
\item[(i)] for $n$ sufficiently large,  $\Pr{\Height(\mathcal{R}_{\lfloor n^{1/9} \rfloor})\geq \Height(\mathcal{R}^{1}_{\lfloor n/2\rfloor})}  \leq  2 c_1 n^{-\frac{1}{8}}$;
\item[(ii)] %there exists a constant $c>0$ such that
for every $n,k \geq 1$, $ \Pr{ |\Height(\mathcal{R}^{2}_{I_n})-\Height(\mathcal{R}^1_{n-I_n})| \geq k } \leq 4 c_1 e^{- c_2 k /4}$.
\end{enumerate}

\end{lem}

The proof of \cref{lem:estimates} is deferred to the end of \cref{subseq:remov-1}.

\begin{proof}[Proof of \cref{prop:increaseheight}(i)] We set $m=\lfloor n^{1/9} \rfloor$ in \eqref{eq:Xb-and-Xb-hat-i}.

We may first assume that  $(\mathcal{T}_{k}(\Xb))_{ 0 \leq k \leq 2m-1}=(\mathcal{T}_{k}(\hat{\Xb}))_{ 0 \leq k \leq 2m-1}$. Observe that $\mathcal{T}_{2m-1}(\Xb)=\mathcal{T}_{2m-1}(\hat{\Xb})$ is a random recursive tree with $m$ edges, where $2$ uniform vertices are active and the other $m-1$  are frozen. We denote by $U_{m}$ and ${V}_{m}$ these two active vertices. Without loss of generality, we may assume that in $\mathcal{T}_{2m}(\Xb)$, $V_m$ is frozen and $U_{m}$ is the only active vertex. Finally, denote by $U'_{m}$ the new active vertex of $\mathcal{T}_{2m+1}(\Xb)$, which is a child of $U_{m}$.

We can then naturally couple $(\mathcal{T}_{k+2}(\Xb))_{ 2m-1 \leq k \leq 2m-1+n}$ and $(\mathcal{T}_{k}(\hat{\Xb}))_{ 2m-1 \leq k \leq 2m-1+n}$ so that the following holds (see \cref{fig:freeze_tree_removal} for an illustration):
\begin{enumerate}
\item[--] $\mathcal{T}(\Xb)$ is obtained by grafting  $\mathcal{R}^{2}_{n-I_n}$  on $U_{m}$ and  $ \mathcal{R} ^{1}_{I_n}$ on $U'_{m}$,
\item[--]$\mathcal{T}(\hat{\Xb})$ is obtained by grafting  $\mathcal{R}^{2}_{n-I_n}$  on $V_{m}$ and  $ \mathcal{R} ^{1}_{I_n}$ on $U_{m}$.
\end{enumerate}
This follows from the well-known fact that in a random recursive tree with $n+1$ edges, when removing the first added edge, the number of edges of the connected component containing the root is uniformly distributed on $\{0,1,\ldots,n\}$.

\begin{figure}[h!]
    \centering
    \includegraphics[scale=1]{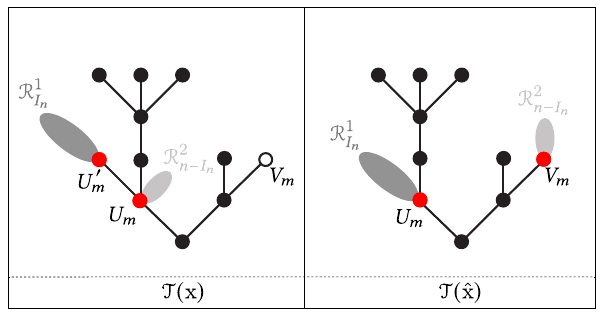}
    \caption{An illustration of the coupling of $\T(\Xb)$ and $\T(\hat{\Xb})$ in the proof of  \cref{prop:increaseheight}(i). The empty vertex is frozen. 
    }
    \label{fig:freeze_tree_removal}
\end{figure}

This coupling implies that the following equalities in distribution hold jointly:

\begin{equation}\label{eq:HeightCoupled}
\begin{aligned}
\Height(\T(\Xb)) &\ed \max\Big(\Height(\mathcal{R}_m),\Ht(U_m)+1+\Height(\mathcal{R}^{1}_{I_n}),\Ht(U_m)+\Height(\mathcal{R}^{2}_{n-I_n})\Big),\\
\Height(\T(\hat{\Xb})) &\ed \max\Big(\Height(\mathcal{R}_m),\Ht(U_m)+\Height(\mathcal{R}^{1}_{I_n}),\Ht(V_m)+\Height(\mathcal{R}^{2}_{n-I_n})\Big).
\end{aligned}
\end{equation}
The intuition behind the remainder of the proof is that $\max(\Ht(U_m),\Ht(V_m))$ is typically larger than $\Ht(U_m)+1$. Formally, we let $\Delta_{n}=\Height(\mathcal{R}^{2}_{n-I_n})-\Height(\mathcal{R}^1_{I_n})$. Using the identity $\max(a,b)=a+\max(0,b-a)$, we decompose the maxima by writing
 \begin{equation}\label{eq:couplingmadeeasy}
\begin{aligned}
\Height(\T(\Xb)) &\ed \Height(\mathcal{R}^{1}_{I_n})+\Ht(U_m)+\max(1,\Delta_{n} )+\max(0,{Z}_{m,n} )\\
\Height(\T(\hat{\Xb})) &\ed \Height(\mathcal{R}^{1}_{I_n})+\Ht(U_m)+\max\big(0,\Ht(V_m)-\Ht(U_m)+\Delta_{n}\big)+ \max\big(0,\hat{Z}_{m,n} \big),
\end{aligned}
\end{equation}
where
\[\begin{aligned}
    Z_{m,n} &=\Height(\mathcal{R}_m)-\max\big(\Ht(U_m)+1+\Height(\mathcal{R}^{1}_{I_n}),\Ht(U_m)+\Height(\mathcal{R}^2_{n-I_n})\big) \ \text{ and } \\ 
    \hat{Z}_{m,n} &=\Height(\mathcal{R}_m)-\max\big(\Ht(U_m)+\Height(\mathcal{R}^{1}_{I_n}),\Ht(V_m)+\Height(\mathcal{R}^2_{n-I_n})\big).
\end{aligned}\]
We claim that both $\Es{\max(0,{Z}_{m,n})} \rightarrow 0$ and $\Es{\max(0,\hat{Z}_{m,n})} \rightarrow0$. Indeed, observe that 
\begin{equation*}
\max\big(0,{Z}_{m,n} \big) \leqst m \cdot \mathbbm{1}_{ \Height(\mathcal{R}_m)\geq \Height(\mathcal{R}^{1}_{\lfloor n/2\rfloor})} \qquad \textrm{and} \qquad \max\big(0,\hat{Z}_{m,n} \big) \leqst m \cdot \mathbbm{1}_{ \Height(\mathcal{R}_m)\geq \Height(\mathcal{R}^{1}_{\lfloor n/2\rfloor})}.
\end{equation*}
since $\Height(\mathcal{R}_m)\leq m$ and 
$\Height(\mathcal{R}^{1}_{\lfloor n/2\rfloor}) \leqst \max\big(\Height(\mathcal{R}^{1}_{I_n}), \Height(\mathcal{R}^2_{n-I_n}) \big)$.
The fact that $\Es{\max(0,{Z}_{m,n})} \rightarrow 0$ and $\Es{\max(0,\hat{Z}_{m,n})} \rightarrow 0$ then follows from \cref{lem:estimates}(i), using  $m=\lfloor n^{1/9} \rfloor$. 

Now observe that  $\Es{|\Delta_{n}|}=O(1)$ by \cref{lem:estimates}(ii). Using the bounds
$$\max\Big(0,\Ht(V_m)-\Ht(U_m)\Big)- |\Delta_{n}|  \leq \max\Big(0,\Ht(V_m)-\Ht(U_m)+\Delta_{n}\Big) \leq \max\Big(0,\Ht(V_m)-\Ht(U_m)\Big)+|\Delta_{n}|$$
combined with \eqref{eq:couplingmadeeasy}, it follows
that
\begin{align*}
    \Es{\Height(\T(\hat{\Xb}))} -\Es{\Height(\T(\Xb))} &= \Es{\max\Big(0,\Ht(V_m)-\Ht(U_m)\Big)}+O(1) \\ 
    &= \tfrac{1}{2}\Es{\big|\Ht(U_m)-\Ht(V_m)\big|}+O(1),
\end{align*}
where the final equality is due to the exchangeability of $(U_m, V_m)$.

To conclude, it suffices to check that $\lim_{m\to\infty}\Es{\big|\Ht(U_m)-\Ht(V_m)\big|}=+\infty$. This is a direct consequence of~\citet[Theorem 2.4]{MR4512393} applied with $k=2$ and $a_1=a_2=0$, which implies that $(\Ht(U_m)-\Ht(V_m))/ \sqrt{\ln(m)}$ converges in distribution to a $ \mathcal{N}(0,2)$ random variable. This
implies that
\begin{equation}
    \label{eq:cvinfinity}
    \Es{\Height(\T(\hat{\Xb}))} -\Es{\Height(\T(\Xb))} \to \infty
\end{equation}
as $n\to \infty$, which completes the proof.
\end{proof}

\begin{proof}[Proof of \cref{lem:estimates}]
We use~\citet[Corollary 1.3]{AF13}, which gives the existence of a constant $c>0$ such that for every $n, k \geq 0$,
\begin{equation}
\label{eq:AF}
\Pr{\left| \Height(\mathcal{R}_n)-\Es{\Height(\mathcal{R}_n)} \right| \geq k} \leq c e^{- \frac{1}{3e}k}.
\end{equation}
For the first assertion, we bound
$\Pr{ \Height(\mathcal{R}_{\lfloor n^{1/9} \rfloor})  \geq \Height(\mathcal{R}^{1}_{\lfloor n/2\rfloor})}$ from above by 
  $$ \Pr{ \Height(\mathcal{R}_{\lfloor n^{1/9} \rfloor})\geq \frac{e}{2}\log(n)}+\Pr{\Height(\mathcal{R}^{1}_{\lfloor n/2\rfloor}) \leq \frac{e}{2}\log(n)}.
$$
By \eqref{eq:H}, for $n$ sufficiently large we have
 $$\left \{ \Height(\mathcal{R}_{\lfloor n^{1/9} \rfloor})\geq \frac{e}{2}\log(n)\right\} \subset \left \{ \Height(\mathcal{R}_{\lfloor n^{1/9} \rfloor}) - \Es{\Height(\mathcal{R}_{\lfloor n^{1/9} \rfloor})} \geq \frac{3e}{8} \log(n)  \right\} $$
 and
  $$\left \{\Height(\mathcal{R}^{1}_{\lfloor n/2\rfloor}) \leq \frac{e}{2}\log(n) \right\} \subset \left \{ \Height(\mathcal{R}_{\lfloor n/2\rfloor})-\Es{\Height(\mathcal{R}_{\lfloor n/2\rfloor})}  \leq - \frac{3e}{8} \log(n)  \right\} .$$
 Bounding both of these using \eqref{eq:AF} yields the desired result in (i).

Now, to simplify notation, we set $H^1_{i}=\Height(\mathcal{R}^{1}_{i})$, $H^{2}_{i}=\Height(\mathcal{R}^{2}_{i})$ and $\mathsf{E}(i)=\Es{\Height(\mathcal{R}_{i})}$ for each $0 \leq i \leq n$, so that $\mathsf{E}(I_n)$ is the random variable $\Es{\Height(\mathcal{R}_{I_n}) \,|\, I_n}$. 
First, the asymptotic expansion \eqref{eq:H} ensures the existence of a constant $C>0$ such that every $n \geq 1$  and $0 \leq i \leq n$, 
\begin{equation}
\label{bound:esp_In-n}
\left|\mathsf{E}(n)-\mathsf{E}(i)\right|\leq C\left|\log \frac{1+i}{1+n}\right|+C= - C\log \frac{1+i}{1+n}+C.
\end{equation}
Next, we write
$$
|\Height(\mathcal{R}^{2}_{I_n})-\Height(\mathcal{R}^1_{n-I_n})| \leq  \left|H^{2}_{I_n}- \mathsf{E}(I_{n}) \right| + \left|H^{1}_{n-I_n}- \mathsf{E}(n-I_{n})\right|
+ \left| \mathsf{E}(I_{n})- \mathsf{E}(n)\right| + \left| \mathsf{E}(n-I_{n})-\mathsf{E}(n)\right|.
$$
Since $I_n$ has the same law as $n-I_n$, we obtain that for every $n,k \geq 1$,
\[\Pr{ |\Height(\mathcal{R}^{2}_{I_n})-\Height(\mathcal{R}^1_{n-I_n})| \geq k }  \leq  2\Pr{\left|H^{1}_{I_n}-  \mathsf{E}(I_{n}) \right| \geq k/4}+ 2 \Pr{\left| \mathsf{E}(I_{n})-\mathsf{E}(n)\right|\geq k/4}.\]
By conditioning on $I_{n}$ in the first term of the right hand side and using \eqref{bound:esp_In-n}, it follows that
\[\Pr{ |\Height(\mathcal{R}^{2}_{I_n})-\Height(\mathcal{R}^1_{n-I_n})| \geq k }  \leq  2\max_{0\leq i\leq n}\Pr{\left|H^{1}_{i}-  \mathsf{E}(i)\right| \geq k/4}+ 2 \Pr{\frac{1+I_n}{1+n}\leq e^{1- k/(4C)}}.\]
Using \eqref{eq:AF} and the fact that $I_n$ is uniform on $\{0, \dots, n\}$, we conclude that
\[
\Pr{ |\Height(\mathcal{R}^{2}_{I_n})-\Height(\mathcal{R}^1_{n-I_n})| \geq k }  \leq  2 c e^{-\frac{1}{12e} k} +2 e^{1- \frac{1}{4C} k}. \qedhere
\]
\end{proof}

\subsection{Removing an attachment step immediately followed by a freezing step can increase the expected height}
\label{subseq:remov+1-1}

Here, we establish \cref{prop:increaseheight}(ii). We set $m=\lfloor n^{1/9} \rfloor$ and
$$\Xb= (+1)^m (-1)^{m-1} (+1,-1)(+1)^n \qquad \textrm{and} \qquad \hat{\Xb}= (+1)^m (-1)^{m-1} (+1)^n.$$
As in the proof of item (i), this proof is also based on a coupling between $\T(\Xb)$ and $\T(\hat{\Xb})$.  Specifically, we may first assume that  $(\mathcal{T}_{k}(\Xb))_{ 0 \leq k \leq 2m-1}=(\mathcal{T}_{k}(\hat{\Xb}))_{ 0 \leq k \leq 2m-1}$. Observe that $\mathcal{T}_{2m-1}(\Xb)=\mathcal{T}_{2m-1}(\hat{\Xb})$ is a random recursive tree where two uniform vertices, denoted by $U_m$ and $V_m$, are active and the other $m-1$ are frozen. Without loss of generality, we can assume that the three active vertices in $\mathcal{T}_{2m}(\Xb)$ are $U_m$, $V_m$ and $U'_m$, with $U'_m$ being a child of $U_m$. 

We keep the notation introduced in \cref{subseq:remov-1}: $(\mathcal{R}_{i})_{i \geq 0}$, $(\mathcal{R}^{1}_{i})_{i\geq 0}$ and $(\mathcal{R}_{i}^{2})_{i\geq 0}$ are three independent  random recursive trees without freezing indexed such that $ \mathcal{R}_{i}$, $ \mathcal{R}^{1}_{i}$ and $ \mathcal{R}^{2}_{i}$ have $i$ edges, and $I_n$ is an independent random variable uniformly distributed on $\{0,\ldots,n\}$. 

To define $\mathcal{T}(\Xb)$, we consider the three disjoint cases depending on which vertex among $U'_m$, $U_m$ and $V_m$ is frozen in $\mathcal{T}_{2m+1}(\Xb)$; see \cref{fig:freeze_tree}:

\begin{enumerate}
\item[(a)] if $U'_m$ is frozen in $\mathcal{T}_{2m+1}(\Xb)$, $\mathcal{T}(\Xb)$ is obtained by grafting  $ \mathcal{R} ^{1}_{I_n}$ on $U_{m}$ and  $\mathcal{R}^{2}_{n-I_n}$  on $V_{m}$,
\item[(b)] if $U_m$ is frozen in $\mathcal{T}_{2m+1}(\Xb)$, $\mathcal{T}(\Xb)$ is obtained by grafting  $ \mathcal{R} ^{1}_{I_n}$ on $U'_{m}$ and  $\mathcal{R}^{2}_{n-I_n}$  on $V_{m}$,
\item[(c)] if $V_m$ is frozen in $\mathcal{T}_{2m+1}(\Xb)$, $\mathcal{T}(\Xb)$ is obtained by grafting   $ \mathcal{R} ^{1}_{I_n}$  on $U'_{m}$ and $\mathcal{R}^{2}_{n-I_n}$ on $U_{m}$.
\end{enumerate}

Define $\mathcal{T}(\hat{\Xb})$ as follows.
\begin{enumerate}
\item[--]$\mathcal{T}(\hat{\Xb})$ is obtained by grafting  $ \mathcal{R} ^{1}_{I_n}$ on $U_{m}$ and  $\mathcal{R}^{2}_{n-I_n}$  on $V_{m}$.
\end{enumerate}

\begin{figure}[h!]
    \centering
    \includegraphics[scale=0.995]{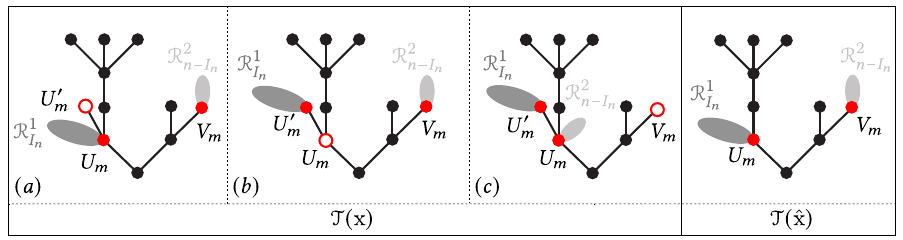}
    \caption{An illustration of the coupling of $\T(\Xb)$ and $\T(\hat{\Xb})$ in the proof of  \cref{prop:increaseheight}(ii).  The empty vertices are frozen.}
    \label{fig:freeze_tree}
\end{figure}
Observe that each of the cases (a)--(c) occurs with equal probability:
$$
\Pr{\textrm{$U'_m$ is frozen in $\mathcal{T}_{2m+1}(\Xb)$}}
=
\Pr{\textrm{$U_m$ is frozen in $\mathcal{T}_{2m+1}(\Xb)$}}
=
\Pr{\textrm{$V_m$ is frozen in $\mathcal{T}_{2m+1}(\Xb)$}}
=
\frac{1}{3}.
$$

\begin{proof}[Proof of \cref{prop:increaseheight}(ii).]
In case (a), if  $I_n\neq0$ then $\mathcal{T}(\Xb)$ and $\mathcal{T}(\hat{\Xb})$ have the same height. Otherwise, their heights differ by at most one. Thus
\begin{equation}\label{eq:case1Pr1/3}
    \Es{\Height(\mathcal{T}(\Xb))\mid \ U'_m \text{ is frozen in }\mathcal{T}_{2m+1}(\Xb)}\leq \Es{\Height(\mathcal{T}(\hat{\Xb}))}+\frac{1}{n+1}.  
\end{equation}

In case (b),  the heights of  $\mathcal{T}(\Xb)$ and $\mathcal{T}(\hat{\Xb})$ differ by at most $1$, so
\begin{equation}\label{eq:case2}
    \Es{\Height(\mathcal{T}(\Xb))\mid \ U_m \text{ is frozen in }\mathcal{T}_{2m+1}(\Xb)}\leq \Es{\Height(\mathcal{T}(\hat{\Xb}))}+1. 
\end{equation}

In case (c), we are in the same framework as the proof of \cref{prop:increaseheight}(ii) (compare with \cref{fig:freeze_tree_removal}), so by \eqref{eq:cvinfinity} we have
\begin{equation}\label{eq:case3}
     \Es{\Height(\mathcal{T}(\Xb))\mid \ V_m \text{ is frozen in }\mathcal{T}_{2m+1}(\Xb)}- \Es{\Height(\mathcal{T}(\hat{\Xb}))}  
     \to
     % \quad \mathop{\longrightarrow}_{n \rightarrow \infty} \quad 
     - \infty 
\end{equation}
as $n \to \infty$. Combining \eqref{eq:case1Pr1/3}, \eqref{eq:case2} and  \eqref{eq:case3} yields
    $$\Es{\Height(\T(\Xb))} -\Es{\Height(\T(\hat{\Xb}))}  
    % \quad \mathop{\longrightarrow}_{n \rightarrow \infty} \quad 
    \to -\infty
    $$
as $n \to\infty$, completing the proof.
\end{proof}

\subsection{Removing an attachment step can increase the expected height}

In this section, we establish \cref{prop:increaseheight}(iii). We now set 
$$\Xb=(+1,+1,-1) (+1)^n\qquad \textrm{and} \qquad \hat{\Xb}=(+1,-1) (+1)^n.$$
We show that $ \Es{\Height(\T(\hat{\Xb}))} >\Es{\Height(\T(\Xb))}$ for $n$ sufficiently large via an appropriate coupling.

To this end, we first introduce some notation. Recall that $ \mathcal{R}_{n}$ denotes a random recursive tree with $n$ edges. We let $T_n^1$ and $T_n^2$ be the subtrees obtained by deleting the edge $(1,2)$ from $\mathcal{R}_n$, such that $T_n^i$ is rooted at $i$ for $i\in\{1,2\}$.  We then couple $ \T(\hat{\Xb})$, $\T({\Xb})$ and $ \mathcal{R}_{n}$ as follows. Observe that $\mathcal{T}(+1,+1,-1)$ and $\mathcal{T}(+1,-1,+1)$ both have $3$ vertices, two of which are active. Then,

\begin{enumerate}
\item[--] $\mathcal{T}(\Xb)$  is obtained from $\mathcal{T}(+1,+1,-1)$ by grafting $T_{{n+1}}^1$ on  its active vertex which appeared first, and by grafting $T_{n+1}^2$ on the other active vertex;
\item[--] $\mathcal{T}(\hat{\Xb})$ is obtained from $\mathcal{T}(+1,-1,+1)$ by grafting $T_{{n}}^1$ on  its active vertex which appeared first, and by grafting $T_{n}^2$ on the other active vertex;
\end{enumerate}
See \cref{fig:config} for an illustration of the coupling.

\begin{figure}
    \centering
    \includegraphics[scale=1]{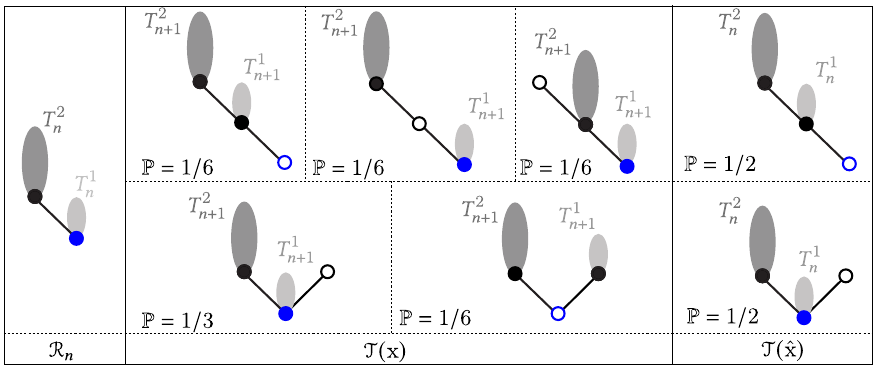}
    \caption{An illustration of the coupling between $ \mathcal{R}_{n}$,  $\T(\Xb)$ and $\T(\hat{\Xb})$ in the proof of \cref{prop:increaseheight}(iii). The root is blue, and the empty vertices are frozen. The probability of each configuration is indicated. }
    \label{fig:config}
\end{figure}

The idea of the proof is to give different expressions for $ \Es{\Height(\T(\hat{\Xb}))}$, $\Es{\Height(\T(\Xb))}$, and $\Es{\Height(\mathcal{R}_{n})}$ using the coupling. First, by the definition of $T^1_{n}$ and $T^{2}_{n}$, we clearly have
\begin{equation}
\label{eq:HRn}
\Height(\mathcal{R}_n)=\max\left( \Height(T_n^1), 1+ \Height(T_n^2)\right).
\end{equation}
Then, by the construction of $\mathcal{T}(\hat{\Xb})$ we have
\begin{align}
\Es{\Height(\mathcal{T}(\hat{\Xb}))}
    &=\frac{1}{2}\Es{\max\left(1+\Height(T_{{n}}^1), 2+ \Height(T_{{n}}^2)\right)} + \frac{1}{2} \Es{\max\left(\Height(T_{{n}}^1), 1+ \Height(T_{{n}}^2)\right)} \notag \\
    &= \Es{\Height(\mathcal{R}_{{n}})}+ \frac{1}{2},  \label{eq:ht-Tx-hat-iii}
\end{align}
where we use \eqref{eq:HRn} for the second equality. Similarly, we have
\begin{align*}
    \Es{\Height(\mathcal{T}(\Xb))}  &= \frac{1}{6}\Es{\max\left(1+\Height(T_{{n+1}}^1), 2+\Height(T_{{n+1}}^2)\right)  } +\frac{1}{6}\Es{\max\left(\Height(T_{{n+1}}^1), 2+\Height(T_{{n+1}}^2)\right)}\\
    & \quad + \frac{1}{6}\Es{\max\left(2, \Height(T_{{n+1}}^1), 1+\Height(T_{{n+1}}^2)\right) } +\frac{1}{3}\Es{\max\left(\Height(T_{{n+1}}^1),1+\Height(T_{{n+1}}^2)\right) }\\
    & \quad + \frac{1}{6}\Es{ \max\left(1+\Height(T_{{n+1}}^1), 1+\Height(T_{{n+1}}^2) \right)}.
\end{align*}
Simplifying gives
\begin{align*}
    \Es{\Height(\mathcal{T}(\Xb))} &= \frac{1}{3}+ \frac{2}{3} \Es{ \max\left(\Height(T_{{n+1}}^1), 1+\Height(T_{{n+1}}^2)\right) }  + \frac{1}{6} \Es{ \max\left(\Height(T_{{n+1}}^1),\Height(T_{{n+1}}^2 )\right)}  \\ 
    &\quad +\frac{1}{6}\Es{\max\left(\Height(T_{{n+1}}^1), 2+\Height(T_{{n+1}}^2)\right)} + \frac{1}{6}\Pr{\max\left(\Height(T_{n+1}^1), 1+\Height(T_{n+1}^2)\right)=1}.
\end{align*}
To compare this quantity with $\Es{\Height(\mathcal{R}_{{n+1}})}$, we introduce the events $A_{n}=\{\Height(T_{{n}}^1)<\Height(T_{{n}}^2)+1\}$, $B_{n}=\{\Height(T_{{n}}^1)=\Height(T_{{n}}^2)+1\}$ and $ C_{n}=\{\Height(T_{{n}}^1)>\Height(T_{{n}}^2)+1\}$.  Then,
$$
\Es{\Height(\mathcal{R}_n)}= \Es{(1+\Height(T_{{n}}^2)) \mathbbm{1}_{A_{n} \cup B_{n}}}+\Es{\Height(T_{{n}}^1) \mathbbm{1}_{C_{n}}}
$$
and
\begin{align*}
\Es{\Height(\mathcal{T}(\Xb))} 
&=
\frac{1}{3}+ \Es{ \left(1+ \Height(T_{{n+1}}^2) \right) \mathbbm{1}_{A_{n+1}}}+ \Es{ \left(1+ \frac{1}{6}+ \Height(T_{{n+1}}^2)\right) \mathbbm{1}_{B_{n+1}}}+ \Es{ \Height(T_{{n+1}}^2) \mathbbm{1}_{C_{n+1}}}\\
&\quad + \frac{1}{6}\Pr{\max \left(\Height(T_{n+1}^1), 1+\Height(T_{n+1}^2)\right)=1}.
\end{align*}
As a consequence, we have
$$\Es{\Height(\mathcal{T}(\Xb))}= \frac{1}{3} + \Es{\Height(\mathcal{R}_{{n+1}})}+\frac{1}{6}\Pr{\Height(T_{{n+1}}^1)=\Height(T_{{n+1}}^2)+1} {+ \frac{1}{6}\Pr{\Height(\mathcal{R}_{n+1})=1}}.$$
Now, since $(T_{{n+1}}^1,T_{{n+1}}^2)$ is exchangeable,
\begin{align*}
2\Pr{\Height(T_{{n+1}}^1)=\Height(T_{{n+1}}^2)+1}  &= \Pr{\Height(T_{{n+1}}^1)=\Height(T_{{n+1}}^2)+1}+\Pr{\Height(T_{{n+1}}^2)=\Height(T_{{n+1}}^1)+1}\\
&=\Pr{\Height(T_{{n+1}}^1)-\Height(T_{{n+1}}^2)\in\{1,-1\}}
\end{align*}
which is at most $1$. Thus, using \eqref{eq:ht-Tx-hat-iii}, we have
\begin{align}
    \Es{\Height(\mathcal{T}(\Xb))} {- \frac{1}{6}\Pr{\Height(\mathcal{R}_{n+1})=1}}
    & 
    \leq \Es{\Height(\mathcal{R}_{{n+1}})}+\frac{1}{3} +\frac{1}{12} \notag \\ 
    &=\Es{\Height(\mathcal{T}(\hat{\Xb}))}-\frac{1}{12}+\Es{\Height(\mathcal{R}_{{n+1}})-\Height(\mathcal{R}_{{n}})}.
\label{eq:HeightComparison}
\end{align}
It is clear that $\Pr{\Height(\mathcal{R}_{n+1})=1}=1/(n+1)!$, which goes to $0$ as $n\to\infty$.
To complete the proof we show that
\begin{equation}
\label{eq:dif}
\Es{\Height(\mathcal{R}_{{n+1}})-\Height(\mathcal{R}_{{n}})} 
\to 0
\end{equation}
as $n \to\infty$. To this end, note that when we add an edge, the height of the tree increases by at most $1$. In fact, it increases if and only if the new vertex connects to a vertex of maximum height, so that
\begin{equation}
\label{eq:HeightIncrease}
    \Es{\Height(\mathcal{R}_{{n+1}})-\Height(\mathcal{R}_{{n}})} =  \Pr{\Height(\mathcal{R}_{{n+1}})=\Height(\mathcal{R}_{{n}})+1}
    = \Pr{\Ht(U_{{n}})=\Height(\mathcal{R}_{{n}})},
\end{equation}
where $U_{{n}}$ is a vertex of $\mathcal{R}_{{n}}$ chosen uniformly at random. It is well known (see, e.g.,~\citet[Theorem O1]{devroye1988applications}) that $\Ht(U_{{n}})$ converges in probability to $\log(n)$, but by \eqref{eq:H}, 
$\Pr{\Height(\mathcal{R}_{{n}}) \geq 1.1 \log(n)} \rightarrow1$ as $n \rightarrow\infty$. This implies \eqref{eq:dif} and completes the proof.

%It is well known (see e.g.~\cite[Section 2.3]{BBKK23+} for a proof) that  if $(B_i)_{i\in \mathbb{N}}$ is a sequence of independent Bernoulli random variables with $\Es{B_{i}} = 1/i$, then
%$$\Ht(U_{{n}}) 
%\ \mathop{=}^{\mathrm{d}} \ 
%\sum_{i=1}^{{n}} B_i.$$
%Then, by Bennett's inequality (\cref{prop:bennett}) we have $\Pr{\Ht(U_{{n}})\geq 1.1 \log(n)}=o(1)$, \simon{a result independently proven in \cite{devroye1988applications}}. 

Note that even though $ \Es{\Height(\T(\hat{\Xb}))} >\Es{\Height(\T(\Xb))}$ for $n$ sufficiently large, $\Height(\T(\hat{\Xb}))$ does not stochastically dominate $\Height(\T(\Xb))$. It is indeed a simple matter to check that
$$\Pr{\Height(\mathcal{T}(\Xb))=1}=\frac{1}{3}\frac{1}{({n+1})!}{<} \frac{1}{2}\frac{1}{{n}!}= \Pr{\Height(\mathcal{T}(\hat{\Xb}))=1}.$$

\section{Removing the first consecutive attachment and freezing steps stochastically decreases the height}
\label{sec:3}

\subsection{An alternative construction of uniform attachment trees with freezing}
\label{sec:3.1}

Consider a choice sequence $\Xb= (\X_i)_{1 \leq i \leq m} \in \{-1,+1\}^m$ such that $S_{j}(\Xb)>0$ for every $j \in \iinterval{1}{m-1}$. Here we recall from \citet{BBKK23+}   an alternative time-reversed construction of uniform attachment trees with freezing, which  can be seen as a growth coalescence process of forests. This is a generalization of the connection between random recursive trees and Kingman's coalescent introduced by~\citet{DR76}.

If $T_{1}$ and $T_{2}$ are two rooted trees, we denote by $T_{2} \rightarrow T_{1}$ the tree   obtained by adding an edge between the roots  of $T_{1}$ and $T_{2}$ and keeping the root of $T_{1}$. We say that $T_{2} \rightarrow T_{1}$ is the tree obtained by grafting $T_{2}$ on the root of $T_{1}$. Finally, we denote by  {$\circled{f}$} the one-vertex rooted tree labelled {$f$}.

\algohead{algo2}
We inductively construct a sequence  $(\F^{m}(\Xb),\F^{m-1}(\Xb), \ldots, \F^{0}(\Xb))$ of forests of rooted vertex-labelled trees, in such a way that for every $0 \leq i \leq m$, the forest $\F^{i}(\Xb)$ is an \emph{ordered}\footnote{As with most coalescent processes, the law of the final forest $\F^0(\Xb)$, which consists of the single tree $\mathsf{T}^0_{m}(\Xb)$, does not depend on the ordering of the ancestral blocks; thus, our decision to work with ordered lists of trees, and our care in appending each new vertex {$\circled{f}$} in the last position of the list, might come as a surprise. In fact, this choice is made solely in preparation for the next subsection, where this framework simplifies the description of a coupling between \cref{algo2} run for two choice sequences $\Xb$ and $\hat{\Xb}$ (\cref{algo3}).} list of $S_{i}(\Xb)$ trees, as follows:
\begin{compitem}
\item Let $\F^{m}(\Xb)$ be a forest made of $S_{m}(\Xb)$ one-vertex rooted trees with labels $a$.
\item For every $1\leq i \leq m$, if $\F^{i}(\Xb)$ has been constructed, define $\F^{i-1}(\Xb)$ as follows:
\begin{compitem}
\item[(a)] if $\X_{i}=-1$, $\F^{i-1}(\Xb)$ is obtained by adding {$\circled{f}$} to $\F^{i}(\Xb)$ in the last position;
\item[(b)] if $\X_{i}=1$, let $(\mathsf{A}_{i},\mathsf{B}_{i})$ be a uniform choice of two different integers in $ \iinterval{1}{S_{i}(\Xb)}$, independently of the previous choices.  Writing $\F^{i}(\Xb)=(T^{i}_{1}, \ldots,T^{i}_{S_{i}(\Xb)})$,
we set
$$
\F^{i-1}(\Xb)=
\begin{cases}
(T^{i}_{1}, \ldots, T^{i}_{\mathsf{A}_{i}-1},T^{i}_{\mathsf{B}_{i}}\rightarrow T^{i}_{\mathsf{A}_{i}} ,T^{i}_{\mathsf{A}_{i}+1}, \ldots, T^{i}_{\mathsf{B}_{i}-1}, T^{i}_{\mathsf{B}_{i}+1}, \ldots , T^{i}_{S_{i}(\Xb)}) & \textrm{ if } \mathsf{A}_{i}<\mathsf{B}_{i},\\
(T^{i}_{1}, \ldots,T^{i}_{\mathsf{B}_{i}-1} ,T^{i}_{\mathsf{B}_{i}+1}, \ldots, T^{i}_{\mathsf{A}_{i}-1}, T^{i}_{\mathsf{B}_{i}}\rightarrow T^{i}_{\mathsf{A}_{i}},T^{i}_{\mathsf{A}_{i}+1}, \ldots,T^{i}_{S_{i}(\Xb)}) & \textrm{ if } \mathsf{A}_{i}> \mathsf{B}_{i}.\\
\end{cases}
$$
 In words, $\F^{i-1}(\Xb)$ is obtained from $\F^{i}(\Xb)$  by removing $T^{i}_{\mathsf{B}_i}$ and replacing $T^{i}_{\mathsf{A}_i}$ with $T^{i}_{\mathsf{B}_{i}}\rightarrow T^{i}_{\mathsf{A}_{i}}$.
\end{compitem}
\item Let $\mathsf{T}^0_{m}(\Xb)$ be the only tree of $\F^0(\Xb)$.
\end{compitem}

\medskip 

By \citet[Theorem 8]{BBKK23+}, the two {rooted vertex-labelled} trees $\T(\Xb)$ and $\mathsf{T}^0_{m}(\Xb)$ have the same law. An important feature of \cref{algo2} is that for every $i \in \iinterval{0}{m}$, the forest $\F^{i}(\Xb)$ is made of $S_{i}(\Xb)$ trees.

\subsection{A coupling for two different sequences}
\label{sec:3.2}

 Consider $\Xb= (\X_i)_{1 \leq i \leq m} \in \{-1,+1\}^m$ such that $S_{j}(\Xb)>0$ for every $j \in \iinterval{1}{m-1}$.  Assume that there exists $k \in \{1,2, \ldots,m-1\}$ such that $\X_1=\X_2=\cdots=\X_{k}=1$ and $\mathrm x_{k+1}=-1$. Let $\hat{\Xb} \in \{-1,+1\}^{m-2}$ be the sequence obtained from $\Xb$ by removing $\mathrm x_k=+1$ and $\mathrm x_{k+1}=-1$.  To simplify notation, for $i \in \iinterval{0}{m}$ we write $\F^{i}$ for $\F^{i}(\Xb)$,  for $i \in \iinterval{0}{m-2}$ we write $\Fh^{i}$ for $\F^{i}(\Xbh)$, and for $i \in \iinterval{0}{k-1}$ we set $S_i=S_i(\Xb)=S_i(\Xbh)$.

The proof of \cref{prop:decreaseheight} is based on a coupling between \cref{algo2} run for $\Xb$ and $\hat{\Xb}$, which intuitively goes as follows. First, for $i=m,m-1, \ldots,k+2$, the random integers in step (b) are chosen to be the same for $\Xb$ and $\hat{\Xb}$. When reading $\X_{k+1}=-1$ in $\Xb$ we add a new ``distinguished vertex'', and when we read $\X_k=+1$ in $\Xb$ we perform a standard grafting step (b). Then, as long as the distinguished vertex has not been used, the trees chosen at grafting steps in $\mathcal{F}(\Xb)$ and $\mathcal{F}(\Xbh)$ will be the same, although with a delay for $\mathcal{F}(\Xbh)$. 

In order to build this coupling it will be convenient to also define a sequence $(\mathsf{M}_i)_{ 0 \leq i \leq k-1}$ of $\{0,1\}$-valued random variables and a sequence $(\mathsf{E}_i)_{0 \leq i \leq k-1}$ of $\N^2 \cup \{\varnothing\}$-valued random variables. Roughly speaking, $(\mathsf{M}_i)_{ 0 \leq i \leq k-1}$ encodes the first time the distinguished vertex has been used in $\mathcal{F}(\Xb)$, and $(\mathsf{E}_i)_{0 \leq i \leq k-1}$ encodes the trees chosen at grafting steps in $\mathcal{F}(\Xb)$ which will be used with a delay in $\mathcal{F}(\Xbh)$. See \cref{fig:coupling} below for an example.

\algohead{algo3}
First, for $i=m,m-1, \ldots,k+2$, the random integers in step (b) of \cref{algo2} are chosen to be the same for   $\Xb$ and $\hat{\Xb}$, so that
 $$
  (\F^{m},\F^{m-1}, \ldots, \F^{k+1})= (\Fh^{m-2},\Fh^{m-3}, \ldots, \Fh^{k-1}).$$

\subparagraph{Initialization step.} We define $\F^{k}$ and $\F^{k-1}$ as follows. Writing $\F^{k+1}=(T^{k+1}_{1}, \ldots, T^{k+1}_{S_{k+1}})$, set
$$\F^{k}=\left({\circled{f}}\,,T^{k+1}_{1}, \ldots, T^{k+1}_{S_{k+1}}\right),$$
where we recall that ${\circled{f}}$ denotes a one-vertex rooted tree labelled {$f$}. In words, $\F^{k}$ is obtained by adding ${\circled{f}}$ to $\F^{k+1}$, but instead of placing it in the last position in the forest as imposed by step (a) of \cref{algo2}, we consider that it is in $0$-th position.  

Then writing $\F^{k}=({\circled{f}}\,,T^{k}_{1}, \ldots, T^{k}_{S_{k-1}})$, build $\F^{k-1}$ as follows. Let $(\mathsf{A}_{k},\mathsf{B}_{k})$ be a uniform choice of two different integers in $ \iinterval{0}{{S_{k-1}}}$, independently of the previous choices. 
\begin{enumerate}
\item[(i)] If $\mathsf{A}_{k}>0$ and $\mathsf{B}_{k}>0$, set $\mathsf{M}_{k-1}=0$, $\mathsf{E}_{k-1}=(\mathsf{A}_{k},\mathsf{B}_{k})$, and perform step (b) in \cref{algo2} by grafting the $\mathsf{B}_{k}$-th tree on the $\mathsf{A}_{k}$-th tree in $\F^{k}$ and by removing the $\mathsf{B}_{k}$-th tree:
$$
\F^{k-1}=
\begin{cases}
({\circled{f}}\,,T^{k}_{1}, \ldots, T^{k}_{\mathsf{A}_{k}-1},T^{k}_{\mathsf{B}_{k}}\rightarrow T^{k}_{\mathsf{A}_{k}} ,T^{k}_{\mathsf{A}_{k}+1}, \ldots, T^{k}_{\mathsf{B}_{k}-1}, T^{k}_{\mathsf{B}_{k}+1}, \ldots, T^{k}_{S_{k-1}}) & \textrm{ if } \mathsf{A}_{k}<\mathsf{B}_{k},\\
({\circled{f}}\,,T^{k}_{1}, \ldots,T^{k}_{\mathsf{B}_{k}-1} ,T^{k}_{\mathsf{B}_{k}+1}, \ldots, T^{k}_{\mathsf{A}_{k}-1}, T^{k}_{\mathsf{B}_{k}}\rightarrow T^{k}_{\mathsf{A}_{k}},T^{k}_{\mathsf{A}_{k}+1}, \ldots,T^{k}_{S_{k-1}}) & \textrm{ if } \mathsf{A}_{k}> \mathsf{B}_{k}.\\
\end{cases}
$$
\item[(ii)] If $\mathsf{A}_k=0$ or $\mathsf{B}_k=0$, set $\mathsf{M}_{k-1}=1$, $\mathsf{E}_{k-1}=(\mathsf{A}_{k},\mathsf{B}_{k})$, and perform step (b) in \cref{algo2}, with the modification that we replace the $\max(\mathsf{B}_k,\mathsf{A}_k)$-th tree with the tree obtained by grafting the $\mathsf{B}_k$-th tree on the root of the $\mathsf{A}_k$-th tree in $\F^{k}$ and by removing the $0$-th tree:
$$
\F^{k-1}=
\begin{cases}
(T^{k}_{1}, \ldots, T^{k}_{\mathsf{B}_{k}-1},T^{k}_{\mathsf{B}_{k}}\rightarrow {\circled{f}}\,,T^{k}_{\mathsf{B}_{k}+1} ,  \ldots , T^{k}_{S_{k-1}}) & \textrm{ if } \mathsf{A}_{k}=0,\\
(T^{k}_{1}, \ldots, T^{k}_{\mathsf{A}_{k}-1},{\circled{f}} \rightarrow T^{k}_{\mathsf{A}_{k}},T^{k}_{\mathsf{A}_{k}+1} ,  \ldots,  T^{k}_{S_{k-1}}) & \textrm{ if }  \mathsf{B}_{k}=0.\\
\end{cases}
$$
\end{enumerate}
Observe that in case (ii), the tree ${\circled{f}}$ in $0$-th position of $\F^{k}$ is always removed.

We shall now recursively build $(\F^{i})_{0 \leq i \leq k-1}$,  $(\Fh^{i})_{0 \leq i \leq k-1}$, $(\mathsf{M}_i)_{ 0 \leq i \leq k-1}$ and  $(\mathsf{E}_i)_{0 \leq i \leq k-1}$ in such a way that for every $j \in \iinterval{0}{k-1}$, there is the tree ${\circled{f}}$ in $0$-th position in $\F^{j}$ if and only if $\mathsf{M}_j=0$.

\subparagraph{Inductive step.} Fix $ j \in \{1,2,\ldots,k-1\}$ and assume that $(\F^{i})_{j \leq i \leq m}$,  $(\Fh^{i})_{j \leq i \leq m-2}$, $(\mathsf{M}_i)_{ j \leq i \leq k-1}$ and  $(\mathsf{E}_i)_{j \leq i \leq k-1}$ have been constructed.

\begin{enumerate}
\item[(A)] If $\mathsf{M}_{j}=1$, set $\mathsf{M}_{j-1}=1$,  $\mathsf{E}_{j-1}=\mathsf{E}_j$. Then let $(\mathsf{A}_{j},\mathsf{B}_{j})$ be a uniform choice of two different integers in $ \iinterval{1}{S_{j}}$, independently of the previous choices and perform step (b) of \cref{algo2}:  graft the $\mathsf{B}_j$-th tree on the $\mathsf{A}_j$-th tree in $\F^{i}$, graft the $\mathsf{B}_j$-th tree on the $\mathsf{A}_j$-th tree in $\Fh^{j}$, and remove the $\mathsf{B}_j$-th trees.
\item[(B)]
If $\mathsf{M}_{j}=0$, writing $\mathsf{E}_j=(\mathsf{a},\mathsf{b})$, $\F^{j}=({\circled{f}}\,,T^{j}_{1}, \ldots, T^{j}_{S_{j}-1})$ and $\Fh^{j}=(\hat{T}^{j}_{1}, \ldots, \hat{T}^{j}_{S_{j}})$,
perform the following actions.   Build $\Fh^{j-1}$ by grafting the $\mathsf{b}$-th tree on the $\mathsf{a}$-th tree in $\Fh^{j}$ and by removing the $\mathsf{b}$-th tree:
$$\Fh^{j-1}=
\begin{cases}
(\hat{T}^{j}_{1}, \ldots, \hat{T}^{j}_{\mathsf{a}-1},\hat{T}^{j}_{\mathsf{b}}\rightarrow \hat{T}^{j}_{\mathsf{a}} ,\hat{T}^{j}_{\mathsf{a}+1}, \ldots, \hat{T}^{j}_{\mathsf{b}-1}, \hat{T}^{j}_{\mathsf{b}+1}, \ldots ,\hat{T}^{j}_{S_{j}}) & \textrm{ if } \mathsf{a}<\mathsf{b},\\
(\hat{T}^{j}_{1}, \ldots,\hat{T}^{j}_{\mathsf{b}-1} ,\hat{T}^{j}_{\mathsf{b}+1}, \ldots, \hat{T}^{j}_{\mathsf{a}-1}, \hat{T}^{j}_{\mathsf{b}}\rightarrow \hat{T}^{j}_{\mathsf{a}},\hat{T}^{j}_{\mathsf{a}+1}, \ldots,\hat{T}^{j}_{S_{j}}) & \textrm{ if } \mathsf{a}> \mathsf{b}.\\
\end{cases}
$$
To build $\F^{j-1}$, $\mathsf{M}_{j-1}$ and $\mathsf{E}_{j-1}$, let $(\mathsf{A}_{j},\mathsf{B}_{j})$ be a uniform choice of two different integers in $ \iinterval{0}{S_{j}-1}$, independently of the previous choices.
\begin{enumerate}
\item[(i)]If $\mathsf{A}_{j}>0$ and $\mathsf{B}_{j}>0$, set $\mathsf{M}_{j-1}=0$, $\mathsf{E}_{j-1}=(\mathsf{A}_j,\mathsf{B}_j)$,  in $\F^{j}$ graft the $\mathsf{B}_j$-th tree on the $\mathsf{A}_j$-th tree in $\F^{j}$ and remove the $\mathsf{B}_j$-th tree:
\end{enumerate}
\medskip
$$
\F^{j-1}=
\begin{cases}
({\circled{f}}\,,T^{j}_{1}, \ldots, T^{j}_{\mathsf{A}_{j}-1},T^{j}_{\mathsf{B}_{j}}\rightarrow T^{j}_{\mathsf{A}_{j}} ,T^{j}_{\mathsf{A}_{j}+1}, \ldots, T^{j}_{\mathsf{B}_{j}-1}, T^{j}_{\mathsf{B}_{j}+1}, \ldots , T^{j}_{S_{j}-1}) & \textrm{ if } \mathsf{A}_{j}<\mathsf{B}_{j},\\
({\circled{f}}\,,T^{j}_{1}, \ldots,T^{j}_{\mathsf{B}_{j}-1} ,T^{j}_{\mathsf{B}_{j}+1}, \ldots, T^{j}_{\mathsf{A}_{j}-1}, T^{j}_{\mathsf{B}_{j}}\rightarrow T^{j}_{\mathsf{A}_{j}},T^{j}_{\mathsf{A}_{j}+1}, \ldots,T^{j}_{S_{j}-1}) & \textrm{ if } \mathsf{A}_{j}> \mathsf{B}_{j}.\\
\end{cases}
$$
\smallskip
\begin{enumerate}
\item[(ii)] If $\mathsf{A}_j=0$ or $\mathsf{B}_j=0$, set $\mathsf{M}_{j-1}=1$, $\mathsf{E}_{j-1}=(\mathsf{A}_j,\mathsf{B}_j)$,  in $\F^{j}$ replace the $\max(\mathsf{B}_j,\mathsf{A}_j)$-th tree with the tree obtained by grafting $\mathsf{B}_j$-th tree on the root of the $\mathsf{A}_j$-th tree  and remove the $0$-th tree:
$$
\F^{j-1}=
\begin{cases}
(T^{j}_{1}, \ldots, T^{j}_{\mathsf{B}_{j}-1},T^{j}_{\mathsf{B}_{j}}\rightarrow {\circled{f}}\,,T^{j}_{\mathsf{B}_{j}+1} ,  \ldots , T^{j}_{S_{j}-1}) & \textrm{ if } \mathsf{A}_{j}=0,\\
(T^{j}_{1}, \ldots, T^{j}_{\mathsf{A}_{j}-1},{\circled{f}} \rightarrow T^{j}_{\mathsf{A}_{j}},T^{j}_{\mathsf{A}_{j}+1} ,  \ldots , T^{j}_{S_{j}-1}) & \textrm{ if }  \mathsf{B}_{j}=0.\\
\end{cases}
$$
\end{enumerate}
\end{enumerate}

{Note that in both steps (A) and (B), $\F^{j-1}$ is built from $\F^j$ by step (b) of \cref{algo2}. Since $x_j=1$ for $j\in\iinterval{1}{k-1}$, the sequence $(\F^m,\ldots,\F^0)$ has the same law as the output of \cref{algo2}, as desired. Similarly, in step (A), $\Fh^{j-1}$ is built from $\Fh^j$ by step (b) of \cref{algo2}. In step (B), the indices $(\mathsf{a},\mathsf{b})$ of the trees involved in the grafting operation to build $\Fh^{j-1}$ are the indices $(\mathsf{A}_{j+1},\mathsf{B}_{j+1})$ of the trees that were involved in the grafting operation to build $\F^{j}$, either during step (i) of the Initialization or during step (B)(i) of the Induction. Now observe that for $j\in\iinterval{1}{k-1}$, we have $\mathsf{M}_j=0$ if and only if $\min(\mathsf{A}_{j+1},\mathsf{B}_{j+1})>0$ and either $j=k-1$ or $\mathsf{M}_{j+1}=0$. Therefore, conditionally given that the algorithm performs step (B) at time $j$, i.e.~that $\mathsf{M}_j=0$,  the indices $(\mathsf{a},\mathsf{b})$ are a uniform choice of two different integers in $\iinterval{0}{S_{j+1}-1}\setminus\{0\}=\iinterval{1}{S_j}$ independently of all the previous choices, as required by step (b) of \cref{algo2}. Hence, \cref{algo3} is indeed a coupling between \cref{algo2} run for $\Xb$ and $\hat{\Xb}$.}

{Finally, the observation that in step (B) we have $(\mathsf{a},\mathsf{b})=(\mathsf{A}_{j+1},\mathsf{B}_{j+1})$ implies, by decreasing induction on $j\in\iinterval{1}{k-1}$, that if $\mathsf{M}_j=0$ then $\F^{j}=({\circled{f}}\,,T^{j}_{1}, \ldots, T^{j}_{S_{j}-1})$ if and only if $\Fh^{j-1}=(T^{j}_{1}, \ldots, T^{j}_{S_{j}-1})$.}

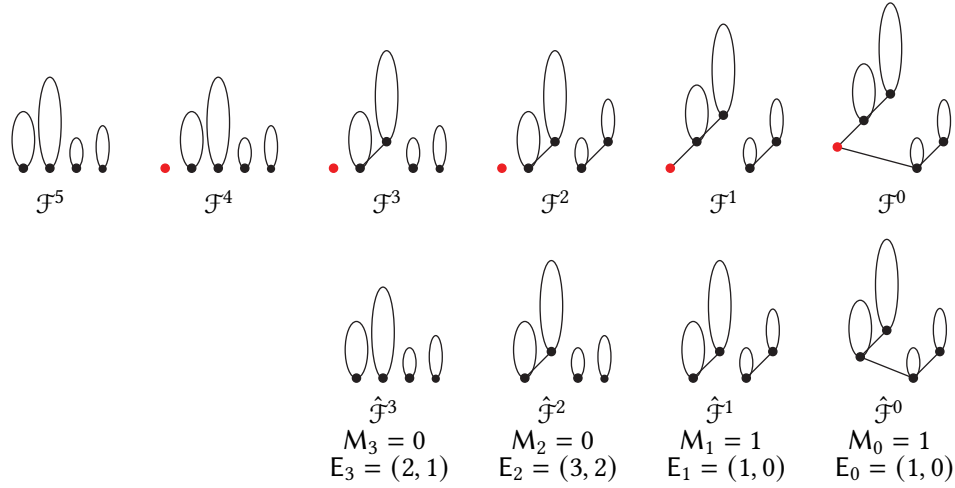
\begin{figure}
    \centering
     \def\myScale{0.5}
     \def\picturesize{50}
\begin{picture}(\picturesize,\picturesize)
\put(0,0){\includegraphics[scale=\myScale]{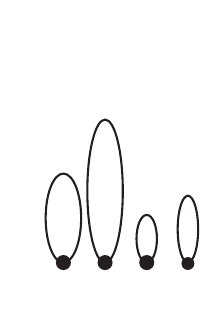}}
\put(20,-5){$\mathcal{F}^{5}$}
\end{picture}
\quad
\begin{picture}(\picturesize,\picturesize)
\put(0,0){\includegraphics[scale=\myScale]{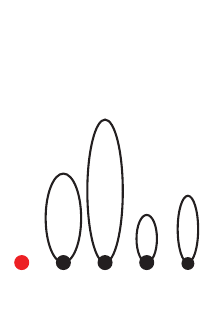}}
\put(20,-5){$\mathcal{F}^{4}$}
\end{picture}
\quad
\begin{picture}(\picturesize,\picturesize)
\put(0,0){\includegraphics[scale=\myScale]{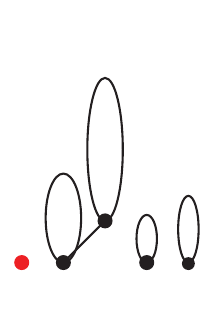}}
\put(20,-5){$\mathcal{F}^{3}$}
\end{picture}
\quad
\begin{picture}(\picturesize,\picturesize)
\put(0,0){\includegraphics[scale=\myScale]{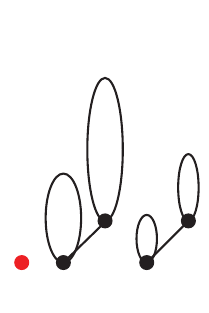}}
\put(20,-5){$\mathcal{F}^{2}$}
\end{picture}
\quad
\begin{picture}(\picturesize,\picturesize)
\put(0,0){\includegraphics[scale=\myScale]{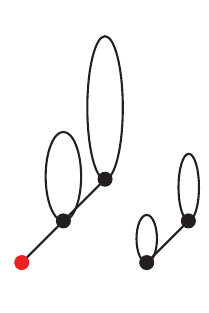}}
\put(20,-5){$\mathcal{F}^{1}$}
\end{picture}
\quad
\begin{picture}(\picturesize,\picturesize)
\put(0,0){\includegraphics[scale=\myScale]{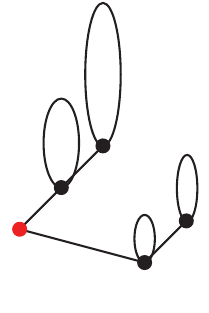}}
\put(20,-5){$\mathcal{F}^{0}$}
\end{picture}\\ \vspace{1cm}
\begin{picture}(\picturesize,\picturesize)
\put(0,0){}
\end{picture}
\quad
\begin{picture}(\picturesize,\picturesize)
\put(0,0){}
\end{picture}
\quad
\begin{picture}(\picturesize,\picturesize)
\put(0,0){\includegraphics[scale=\myScale]{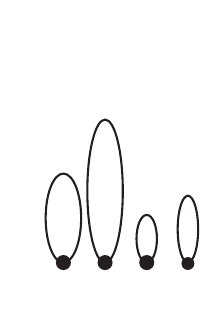}}
\put(20,-5){$\hat{\mathcal{F}}^{3}$}
\put(10,-15){$ \mathsf{M}_{3}=0$}
\put(5,-25){$ \mathsf{E}_{3}=(2,1)$}
\end{picture}
\quad
\begin{picture}(\picturesize,\picturesize)
\put(0,0){\includegraphics[scale=\myScale]{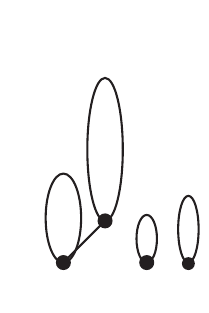}}
\put(20,-5){$\hat{\mathcal{F}}^{2}$}
\put(10,-15){$ \mathsf{M}_{2}=0$}
\put(5,-25){$ \mathsf{E}_{2}=(3,2)$}
\end{picture}
\quad
\begin{picture}(\picturesize,\picturesize)
\put(0,0){\includegraphics[scale=\myScale]{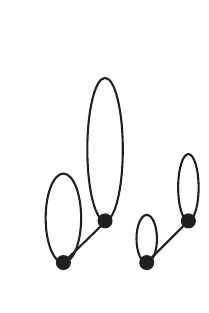}}
\put(20,-5){$\hat{\mathcal{F}}^{1}$}
\put(10,-15){$ \mathsf{M}_{1}=1$}
\put(5,-25){$ \mathsf{E}_{1}=(1,0)$}
\end{picture}
\quad
\begin{picture}(\picturesize,\picturesize)
\put(0,0){\includegraphics[scale=\myScale]{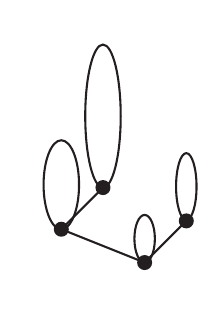}}
\put(20,-5){$\hat{\mathcal{F}}^{0}$}
\put(10,-15){$ \mathsf{M}_{0}=1$}
\put(5,-25){$ \mathsf{E}_{0}=(1,0)$}
\end{picture}
\vspace{1cm}
\caption{Example of the coupling for $k=4$, $\Xb=(+1,+1,+1,+1,-1,\ldots)$ and $\Xbh= (+1,+1,+1,\ldots)$. The vertex in red is the vertex  ${\protect\circled{f}}$ added in the $0$-th position in the initialization step.}
    \label{fig:coupling}
\end{figure}

The proof of \cref{prop:decreaseheight} is now effortless.

\begin{proof}[Proof of \cref{prop:decreaseheight}]
We consider the coupling introduced just before. Set  $j_1= \max \{j \in \iinterval{0}{k-1} : \mathsf{M}_j =1\}$ and observe that $j_1$ is well defined: since $S_1=2$, we must have $\mathsf{M}_0=1$. Also set $\mathsf{E}_{j_{1}}=(\mathsf{a},\mathsf{b})$, $\F^{j_{1}}=(T^{j_{1}}_{1}, \ldots, T^{j_{1}}_{S_{j_{1}}})$ and $\Fh^{j_1}=(\hat{T}^{j_{1}}_{1}, \ldots, \hat{T}^{j_{1}}_{S_{j_{1}}})$. By construction, observe that the following properties hold:
\begin{enumerate}
\item[--] $\mathsf{a}=0$ or $\mathsf{b}=0$; 
\item[--] we have $T^{j_{1}}_{i}=\hat{T}^{j_{1}}_{i}$ for every $i \in \iinterval{1}{S_{j_{1}}} \backslash \{\max(a,b)\}$;
\item[--] the tree $T^{j_{1}}_{\max(a,b)}$ is equal to $\hat{T}^{j_{1}}_{\max(a,b)}\rightarrow{\circled{f}}$ or to ${\circled{f}} \rightarrow \hat{T}^{j_{1}}_{\max(a,b)}$.
\end{enumerate}
As a consequence, for every $i \in \iinterval{1}{S_{j_1}}$, the height of the $i$-th tree of $\F^{j_1}$ is at least equal to the height of the $i$-th tree  of $\Fh^{j_1}$. By step (A), this property remains true for forests $\F^{j}$ and $\Fh^{j}$ for every $j=j_1-1, \ldots,0$. By taking $j=0$ we get the desired result.
\end{proof}

\subsection{Proof of \cref{thm:main}}
\label{sec:3.3}

\begin{proof}[Proof of \cref{thm:main}.]
For $\Xb \in \mathbb{X}_n$, we distinguish two cases depending on $\max_{1 \leq j \leq m} S_{j}(\Xb)$, which represents the maximum number of active vertices reached at some point. Set  $r_{n}=\lfloor n/(3 \log n) \rfloor$ for $n$ large enough and
\begin{align*}
\mathbb{X}^{1}_n &= \{\Xb= (\X_i)_{1 \leq i \leq m} \in  \mathbb{X}_n :  \max_{1 \leq j \leq m} S_{j}(\Xb) \geq r_{n} \},\\
\mathbb{X}^{2}_n &= \{\Xb= (\X_i)_{1 \leq i \leq m} \in  \mathbb{X}_n :  \max_{1 \leq j \leq m} S_{j}(\Xb) < r_{n} \}.
\end{align*}

\paragraph{High values.} Take $\Xb= (\X_i)_{1 \leq i \leq m} \in  \mathbb{X}^{1}_n$. By iteratively removing the first consecutive $+1,-1$'s we can transform $\Xb$ into a sequence $\Yb$ starting with $r_{n}$ times $+1$. By \cref{prop:decreaseheight}, this shows that
\begin{equation}
\label{eq:case1}
\Height(\T(\Xb)) \geqst  \Height(\mathcal{R}_{r_{n}}). 
\end{equation}
Then by \eqref{eq:H} we have 
$$\Es{\Height(\T(\Xb))}-e \log n \geq \Es{\Height(\mathcal{R}_{r_{n}})}-e \log n=e \log \left(\frac{r_n}{n}\right) - \frac{3}{2} \log \log r_{n} +O(1) \geq - 5 \log \log n$$
for $n$ sufficiently large.
Similarly,
\begin{multline*}
\Pr{\Height(\T(\Xb)) \geq e \log(n) - 5 \log \log n}  \\
\geq \Pr{\Height(\mathcal{R}_{r_{n}}) -e \log r_n + \frac{3}{2} \log \log r_{n} \geq e \log(n) - 5 \log \log n - e \log r_{n} + \frac{3}{2} \log \log r_{n}},
\end{multline*} 
which by \eqref{eq:H} goes to $1$ as $ n \rightarrow \infty$ uniformly in $\Xb \in  \mathbb{X}^{1}_n$, since $ e \log(n) - 5 \log \log n - e \log r_{n} + \frac{3}{2} \log \log r_{n} \rightarrow -\infty$.

\paragraph{Low values.}  Take $\Xb= (\X_i)_{1 \leq i \leq m} \in  \mathbb{X}^{2}_n $. In this case, we lower bound $\Height(\T(\Xb))$ by the height of a uniform active vertex of $\T_{m}(\Xb)$, which by \citet[Eq.~(5)]{BBKK23+} has the same law as
$\sum_{i=1}^{m}Y_{i} \mathbbm{1}_{\Xb_{i}=1}$,
where $(Y_{i})_{1 \leq i \leq m}$ are independent Bernoulli random variables of respective parameters $(1/S_{i}(\Xb))_{1 \leq i \leq m}$. 
Since $\# \{1 \leq i \leq m : \X_{i}=1\} =n$ by definition of $ \mathbb{X}_n$, it follows that  
$$\Height(\T(x)) \geqst \mathsf{Bin}(n,1/r_{n}),
$$
where $\mathsf{Bin}(n,1/r_{n})$ is a Binomial random variable with $n$ trials and success probability $1/r_n$.
Then, 
$$\Es{\Height(\T(\Xb))}- e \log n \geq {n/r_n-e\log n \geq } (3-e) \log (n)  \geq - 5 \log \log n.
$$

Next, we use the following concentration inequality for Bernoulli random variables due to Bennett;
see e.g.~\citet[Theorem 2.9]{boucheronconcentration}. We state it below, tailored for our purpose.

\begin{prop}[Bennett's inequality]
\label{prop:bennett}
Let $(Y_{i})_{1 \leq i \leq n}$ be independent Bernoulli random variables of respective parameters $(p_{i})_{1\leq i \leq n}$. Set $m_{n}=\sum_{i=1}^{n} p_{i}$ and assume that $m_{n}>0$. Then for every $t>0$:
\[\Pr{\sum_{i=1}^{n} Y_{i} >m_{n}+ t} \leq \exp \left( - m_{n} g \left( \frac{t}{m_{n}} \right) \right) \qquad \textrm{and} \qquad \Pr{\sum_{i=1}^{n} Y_{i}< m_{n}- t} \leq \exp \left( - m_{n} g \left( \frac{t}{m_{n}} \right) \right),\]
 where $g(u)=(1+u)\ln(1+u)-u$ for $u>0$.
\end{prop}

Hence, setting $s_{n}=n/r_{n}$, Bennett's inequality gives that for every $t > 0$, it holds that $ \Pr{\mathsf{Bin}(n,1/r_{n}) < s_{n}-t } \leq \exp(- s_{n}g(t/s_{n}))$ with $g(u)=(1+u) \ln(1+u)-u$. Applying this inequality with $t={s_n-e\log n }\geq (3-e) \log n$ yields that
$$
\Pr{\Height(\T(\Xb)) < e \log(n)  - 5 \log \log n}  \leq \Pr{\Height(\T(\Xb)) < e \log(n) }  \leq  \exp (-3 \log (n) g(1-e/3)),
$$
which goes to $0$ as $ n \rightarrow \infty$ uniformly in $\Xb \in  \mathbb{X}^{2}_n$.
\end{proof}

\section{Open questions}
\label{sec:open}

We conclude by mentioning several natural directions for future research.
\begin{enumerate}
\item[(1)] \cref{thm:main} shows that for every $\Xb \in \mathbb{X}_n$, $\Height(\T(\Xb))$ is asymptotically at least of the same order as $\Height(\mathcal R_n)$. It is thus natural to ask: do we have
$$ \inf_{n \geq 3} \left( \min_{\Xb \in \mathbb{X}_n}\Es{\Height(\T(\Xb))}  - \Es{\Height(\mathcal{R}_n)} \right) \geq 0 \ ?$$
This would require new ideas and we leave this for future work.
\item[(2)] We have seen in \cref{sec:RRT_vs_alternate} that for $\Xb \in \mathbb{X}_n$,  $\Height(\T(\Xb))$ does not dominate $\Height(\mathcal R_n)$ because of small heights. What is the  
 smallest $h_n$ such that
\[\Height(\mathcal{R}_n) \preceq_{\mathrm{st}} \max(h_n,\Height(\T(\Xb)))\]
for every $\Xb\in \mathbb{X}_n$? An extension of the argument of \cref{sec:RRT_vs_alternate} shows that $h_n\to\infty$.
\item[(3)] What is the smallest $N_n$ such that
\[\Height(\mathcal{R}_n) \preceq_{\mathrm{st}} \Height(\T(\Xb))\]
for every $\Xb\in\mathbb{X}_{N_n}$? %The argument of \cref{sec:RRT_vs_alternate} shows that $ \liminf_{n \rightarrow \infty} N_n\ /(n\log_2 n) \geq 1$.
The argument of \cref{sec:RRT_vs_alternate} shows that $ \liminf_{n \rightarrow \infty} N_n\ /(n\log_2 n) \geq 1$. Indeed, following the argument of \cref{sec:RRT_vs_alternate}, we must have $\Pr{\Height(\mathcal{R}_n) = 1} \geq \Pr{\Height(\mathcal{A}_{N_n}) = 1}$, where $\Pr{\Height(\mathcal{R}_n) =1} = {1}/ n!$ and  $\Pr{\Height(\mathcal{A}_{N_n} )= 1}= {1}/{2^{N_n-1}}$, which implies the claim.
\end{enumerate}

\nocite{*}
\bigskip  
\bibliographystyle{plainnat}
\bibliography{biblio}

\end{document}